\documentclass[11pt]{amsart}
\usepackage{amsmath,amsfonts,amsthm,mathrsfs}
\usepackage{amssymb}

\usepackage[unicode,bookmarks]{hyperref}
\usepackage[usenames,dvipsnames]{xcolor}
\hypersetup{colorlinks=true,citecolor=NavyBlue,linkcolor=Brown,urlcolor=Orange}

\usepackage{tikz-cd,adjustbox}
\usepackage[arrow,curve,matrix]{xy}

\newif\ifdraft
\draftfalse

\counterwithin{equation}{subsection}
\counterwithout{subsection}{section}
\counterwithin{figure}{subsection}

\newtheorem{theorem}[equation]{Theorem}
\newtheorem*{theorem*}{Theorem}
\newtheorem{lemma}[equation]{Lemma}
\newtheorem*{lemma*}{Lemma}
\newtheorem{corollary}[equation]{Corollary}
\newtheorem{proposition}[equation]{Proposition}
\newtheorem*{proposition*}{Proposition}
\newtheorem{conjecture}[equation]{Conjecture}

\theoremstyle{definition}
\newtheorem{definition}[equation]{Definition}
\newtheorem*{definition*}{Definition}
\newtheorem{remark}[equation]{Remark}

\newtheorem{example}[equation]{Example}
\newtheorem*{example*}{Example}
\newtheorem*{problem*}{Problem}

\theoremstyle{plain}

\newcommand{\C}{\mathcal C}

\newcommand{\E}{\mathcal E}
\newcommand{\F}{\mathcal F}

\renewcommand{\H}{\mathcal H}

\renewcommand{\O}{\mathcal O}

\newcommand{\bR}{{\bf R}}

\newcommand{\CC}{\mathbb C}

\newcommand{\HH}{\mathbb H}

\newcommand{\PP}{\mathbb P}
\newcommand{\QQ}{\mathbb Q}

\newcommand{\ZZ}{\mathbb Z}

\newcommand{\DDD}{\mathbb D}

\newcommand{\DD}{\mathbf D}
\newcommand{\derR}{\mathbf R}

\newcommand{\xto}{\xrightarrow} 

\newcommand{\RHom}{{\bf R} \H om}
\newcommand*{\sHom}{\mathscr{H}\kern -.5pt om}
\DeclareMathOperator{\Ext}{Ext}
\newcommand{\sExt}{\E xt} 

\newcommand{\DB}{\underline{\Omega}} 

\DeclareMathOperator{\Gr}{Gr}
\DeclareMathOperator{\DR}{DR}

\DeclareMathOperator{\sing}{sing}
\DeclareMathOperator{\supp}{supp}

\DeclareMathOperator{\lcd}{lcd}
\DeclareMathOperator{\codim}{codim}

\DeclareMathOperator{\depth}{depth}

\DeclareMathOperator{\IC}{IC}

\DeclareMathOperator{\gr}{{\rm gr}}

\DeclareMathOperator{\lcdef}{lcdef}

\newcommand{\theoremref}[1]{\hyperref[#1]{Theorem~\ref*{#1}}}
\newcommand{\lemmaref}[1]{\hyperref[#1]{Lemma~\ref*{#1}}}
\newcommand{\definitionref}[1]{\hyperref[#1]{Definition~\ref*{#1}}}
\newcommand{\propositionref}[1]{\hyperref[#1]{Proposition~\ref*{#1}}}
\newcommand{\conjectureref}[1]{\hyperref[#1]{Conjecture~\ref*{#1}}}
\newcommand{\corollaryref}[1]{\hyperref[#1]{Corollary~\ref*{#1}}}
\newcommand{\exampleref}[1]{\hyperref[#1]{Example~\ref*{#1}}}

\makeatletter
\let\old@caption\caption
\renewcommand*{\caption}[1]{%
	\setcounter{figure}{\value{equation}}%
	\stepcounter{equation}%
	\old@caption{#1}\relax%
}
\makeatother

\newcounter{intro}

\newtheorem{intro-conjecture}[intro]{Conjecture}
\newtheorem{intro-corollary}[intro]{Corollary}
\newtheorem{intro-theorem}[intro]{Theorem}
\newtheorem{intro-proposition}[intro]{Proposition}

\begin{document}

\title[Du Bois complexes of isolated singularities]{Injectivity and vanishing for the Du Bois complexes of isolated singularities}

\author[M.~Popa]{Mihnea~Popa}
\address{Department of Mathematics, Harvard University, 
1 Oxford Street, Cambridge, MA 02138, USA} 
\email{{\tt mpopa@math.harvard.edu}}

\author[W.~Shen]{Wanchun~Shen}
\address{Department of Mathematics, Harvard University, 
1 Oxford Street, Cambridge, MA 02138, USA} 
\email{{\tt wshen@math.harvard.edu}}

\author[A. D.~Vo]{Anh~Duc~Vo}
\address{Department of Mathematics, Harvard University, 
1 Oxford Street, Cambridge, MA 02138, USA} 
\email{{\tt ducvo@math.harvard.edu}}

\thanks{The authors were partially supported by the NSF grant DMS-2040378, and MP by DMS-2401498 as well.}

\subjclass[2020]{}

\begin{abstract}
We prove an injectivity theorem for the cohomology of the Du Bois complexes of varieties with isolated singularities. We use this to deduce vanishing statements for the cohomologies of higher Du Bois complexes of such varieties. Besides some extensions and conjectures in the non-isolated case, we also provide analogues for 
intersection complexes.
\end{abstract}

\maketitle

\makeatletter
\newcommand\@dotsep{4.5}
\def\@tocline#1#2#3#4#5#6#7{\relax
  \ifnum #1>\c@tocdepth 
  \else
    \par \addpenalty\@secpenalty\addvspace{#2}%
    \begingroup \hyphenpenalty\@M
    \@ifempty{#4}{%
      \@tempdima\csname r@tocindent\number#1\endcsname\relax
    }{%
      \@tempdima#4\relax
    }%
    \parindent\z@ \leftskip#3\relax
    \advance\leftskip\@tempdima\relax
    \rightskip\@pnumwidth plus1em \parfillskip-\@pnumwidth
    #5\leavevmode\hskip-\@tempdima #6\relax
    \leaders\hbox{$\m@th
      \mkern \@dotsep mu\hbox{.}\mkern \@dotsep mu$}\hfill
    \hbox to\@pnumwidth{\@tocpagenum{#7}}\par
    \nobreak
    \endgroup
  \fi}
\def\l@section{\@tocline{1}{0pt}{1pc}{}{\bfseries}}
\def\l@subsection{\@tocline{2}{0pt}{25pt}{5pc}{}}
\makeatother



\section{Introduction}
Let $X$ be a complex algebraic variety of dimension $n$, and for each $k \ge 0$ let $\DB_X^k$ be the $k$-th associated graded term of the filtered de Rham complex $\DB_X^\bullet$ with respect to the Hodge filtration, also called the $k$-th Du Bois complex of $X$. Given their growing importance in the study of singularities via Hodge theory, it has become essential to understand the finer homological properties of these complexes. This paper addresses the case of isolated singularities, by focusing on vanishing theorems for the cohomologies of Du Bois complexes, and injectivity theorems for the cohomologies of their duals.

\subsection*{Vanishing of cohomologies}
By definition, the Du Bois complexes $\DB_X^k$ are supported in degrees in the range $[0, n]$. Something better is in fact true; 
a well-known result of Steenbrink \cite[(4.1)]{Steenbrink-vanishing} says that
\[\H^i\DB_X^k=0 \,\,\,\,\,\,{\rm for ~all}\,\,\,\, i> n -k,\]
and without further assumptions on $k$ or the singularities, this is optimal. 

The vanishing of other higher cohomologies $\H^i\DB_X^k$ in the possible non-vanishing range is one way to measure how good the singularities of $X$ are. One of our main goals is to describe concrete conditions under which this holds. Previous results in this direction were obtained in \cite{MOPW} for hypersurfaces, and more generally in \cite{MP-LC} for local complete intersections. Here we address this in the case of isolated singularities, providing some results in the non-isolated case along the way as well.

The appropriate language for studying this problem is that of \emph{higher Du Bois singularities}, studied in \cite{MOPW}, \cite{JKSY}, 
\cite{MP-LC}, \cite{SVV} among others. At least in the local complete intersection case, this condition means that $\DB_X^p$ can be identified with 
the sheaf of K\"ahler differentials $\Omega_X^p$ for certain $p$. Along the way towards a general definition, the weaker notion of \emph{pre-$k$-Du Bois} singularities was introduced in \cite{SVV}; this simply means the vanishing of higher cohomologies, i.e. $\H^i \DB_X^p = 0$ for all $i >0$ and $p \le k$.
Cf. also \cite{Tighe}, using different terminology.

For fixed $k$, the question of whether the cohomology sheaf $\H^i \DB_X^k$ is zero for some $i > 0$ is therefore by definition nontrivial only if 
$X$ is not pre-$k$-Du Bois. Our main vanishing result studies the first degree where this is the case. Note that the answer is influenced by the algebraic properties
of the $0$-th cohomology sheaf $\H^0\DB^k_X$, which is known to always be torsion-free.

\begin{intro-theorem}\label{thm:vanishing-isolated}
 Let $X$ be a variety with pre-$(k-1)$-Du Bois isolated singularities. Then:
 \[\H^i\DB^k_X = 0 \quad \text{ {\rm for} } \quad 0<i<\depth\H^0\DB^k_X -1.\] 
\end{intro-theorem}

A reformulation of this theorem in terms of resolution of singularities and higher direct images of logarithmic forms can be found in Remark \ref{rmk:vanishing-log-poles}.

When $k = 0$, when the hypothesis means that there are no assumptions on $X$, this is due to Steenbrink \cite[Proposition 1]{Steenbrink-DB-invariants}. Using our method of proof however, combined with an injectivity result from \cite{KS-injectivity}, we can extend this case to arbitrary singular sets.\footnote{We will see that for $k = 0$ one can safely replace $\H^0 \DB_X^0$ by $\O_X$.}

\begin{intro-corollary}\label{cor:CM}
If a variety $X$ is (pre-)Du Bois away from a closed subset of dimension $s$, then 
    \[\H^i\DB^0_X = 0 \quad \text{ {\rm for} } \quad 0<i<\depth \O_X- s -1.\]
\end{intro-corollary}

A quick consequence is the fact that a Cohen-Macaulay variety of dimension $n$, with isolated singularities, is Du Bois if and only if the natural morphism 
$H^n (X, \O_X ) \to H^n (X, \DB_X^0)$ is injective (hence an isomorphism); see Corollary \ref{cor:DB-CM-criterion}.

The picture of vanishing results for higher cohomology sheaves is completed by the following ``sliding" rule, which is quite simple but seems to not have been noted before; it holds with no assumption on the singular locus.

\begin{intro-proposition}\label{prop:borderline-vanishing}
Let $X$ be an $n$-dimensional variety. If $k < n$ and $\H^{n-p -1}\DB_X^p=0$ for all $p \le k -1$, then 
\[\H^{n-k}\DB_X^k=0.\]
In particular $\H^n \DB_X^0 = 0$, and more generally if $X$ is pre-$(k-1)$-Du Bois, with $k < n$, then $\H^{n-k}\DB_X^k=0$.
\end{intro-proposition}

Further vanishing results and conjectures in the non-isolated case are discussed later in the Introduction. Note in particular Conjecture \ref{conj:vanishing-general} for an extension of Theorem \ref{thm:vanishing-isolated} to the general case.

\subsection*{Injectivity for the cohomologies of duals}
The key technical result of the paper,  used in the proof of Theorem \ref{thm:vanishing-isolated} as well as for other applications, is the following injectivity theorem for the cohomologies of the dual of the first Du Bois complex that is not a sheaf.

\begin{intro-theorem}\label{thm:injectivity-H0DB}
    Let $X$ be a variety with isolated pre-$(k-1)$-Du Bois singularities. Then the morphism
    \[\RHom_{\O_X}(\DB_X^k,\omega_X^{\bullet})\to \RHom_{\O_X}(\H^0\DB_X^k,\omega_X^{\bullet})\]
    in $\DD^b_{\rm coh} (X)$, obtained by dualizing the canonical map $\H^0\DB_X^k \to \DB_X^k$,
    is injective on cohomology.\footnote{Here $\omega_X^{\bullet}$ is the dualizing complex of $X$.}
\end{intro-theorem}

In Conjecture \ref{conj:injectivity-H0DB} we predict that the statement should hold even without the isolated singularities hypothesis.

An injectivity theorem of this type first appeared in the inspiring paper \cite{KS-injectivity} by Kov\'acs-Schwede for $k=0$.  It was then reinterpreted in terms of the Hodge filtration on local cohomology, and extended to arbitrary $k$ in the case of local complete intersections, 
in \cite{MP-LC} and \cite{MP-lci}. Fundamentally, such injectivity theorems are degeneration at $E_1$ phenomena for appropriate Hodge-theoretic objects, and are now understood to be one of the most essential properties of Du Bois complexes.

Note a subtlety: when $k = 0$, or when $X$ is LCI, K\"ahler differentials are rather well behaved under the $(k-1)$-Du Bois hypothesis, and therefore the right hand side in the previously known injectivity theorems is expressed in terms of $\Omega_X^k$; see e.g. Theorem \ref{thm:MP-injectivity}. This is not the case anymore for arbitrary singularities, where we found that the natural formulation of injectivity is in terms of $\H^0 \DB_X^k$.
Nevertheless, one can deduce from Theorem \ref{thm:injectivity-H0DB} statements about K\"ahler differentials as well. Here we only include 
a special case that is easier to state, while the general result is Corollary \ref{cor:injectivity-MP-gen}; the local cohomological defect ${\rm lcdef}(X)$ is defined in the next subsection, and the definition of $k$-Du Bois singularities is recalled in Section \ref{sec:higher-singularities}.

\begin{intro-corollary}
 Let $X$ be a variety with isolated $(k-1)$-Du Bois singularities, with $\dim X \ge 2$ and ${\rm lcdef} (X) = 0$.  Then the map
        \[\RHom_{\O_X}(\DB_X^k,\omega_X^{\bullet})\to \RHom_{\O_X}(\Omega_{X, {\rm tf}}^k,\omega_X^{\bullet})\]
 is injective on cohomology. Here $\Omega_{X, {\rm tf}}^k$ denotes the quotient of $\Omega_X^k$ by its torsion subsheaf.
 \end{intro-corollary}   

As for the proof of Theorem \ref{thm:injectivity-H0DB}, it is immediate to reduce to the case of projective varieties, in which case we show a more general fact, namely that the statement holds when $X$ is pre-$(k-1)$-Du Bois with possibly 
higher dimensional singular locus, but pre-$k$-Du Bois except at finitely many points; see Theorem \ref{thm:injectivity-H0DB-gen}. We use the degeneration at $E_1$ of the Du Bois version of the Hodge-to-de Rham spectral sequence, inspired by the approach in \cite{KS-injectivity} rather than that in \cite{MP-lci} (as we do not yet have a good theory of the Hodge filtration on local cohomology in the non-LCI case).

Another application of Theorem \ref{thm:injectivity-H0DB} is a very quick proof of the known fact that $k$-rational singularities are $k$-Du Bois, in the case of isolated singularities. We show this in Section \ref{scn:krat-kDB}, where we also recall what is known in this direction.

\subsection*{More on vanishing}
Going back to vanishing statements, our approach also provides a somewhat weaker statement for non-isolated singularities, which is essentially due to formal homological algebra. 

\begin{intro-proposition}\label{prop:vanishing-general}
 Let $X$ be a variety that is pre-$k$-Du Bois away from a closed subset of dimension $s$. Then:
 \[\H^i\DB^k_X = 0 \quad \text{ {\rm for} } \quad 0<i< m_k - s -1,\]
 where $m_k := {\rm min}~ \{\depth\H^0\DB^k_X, ~n - k - {\rm lcdef}(X)+ 1\}$.
\end{intro-proposition}

Here ${\rm lcdef} (X)$ denotes the \emph{local cohomological defect} of $X$, defined as 
$${\rm lcdef}(X)= \lcd (X, Y) - \codim_Y (X),$$ 
where $X \subseteq Y$ is an embedding in a smooth variety, with local cohomological dimension $\lcd (X, Y)$. It does not depend on the embedding, and ${\rm lcdef}(X)= 0$ for local complete intersections, but also for other interesting classes of varieties; for more details see Section \ref{scn:lcd}.  One can make sense of the depth of an object in the derived category, and what is really proven in Proposition \ref{prop:vanishing-general} is a consequence of a general fact about arbitrary such objects, namely the same statement but with $m_k$ replaced by the more abstract
$$m_k^\prime := {\rm min} \{\depth\H^0\DB^k_X, ~  \depth \DB^k_X + 1\}.$$
One then uses  one of the main results of \cite{MP-LC}, which implies that 
$$\depth \DB^k_X \ge n - k - {\rm lcdef} (X),$$
with equality for some $k$. 
To prove Theorem \ref{thm:vanishing-isolated}, we need to combine this abstract formulation with the main injectivity result, Theorem \ref{thm:injectivity-H0DB}.

We also explain in Example \ref{ex:vanishing-lci} how this formal statement can be used when $X$ is a local complete intersection in order to recover 
 \cite[Corollary 13.9]{MP-LC} (see also \cite{MOPW} for hypersurfaces), stating that if $s = \dim X_{{\rm sing}}$, then for all $k$ we have
\[\H^i\DB^k_X = 0 \quad \text{ {\rm for} } \quad 0<i< n - k - s -1.\]

At the end of Section \ref{scn:vanishing} we give examples where certain intermediate cohomologies do not vanish, showing the failure of some possible extensions of this result to the general setting.

\subsection*{Conjectures}
The main results of this paper are likely to admit natural extensions to the case of non-isolated singularities. The most important 
extends Theorem \ref{thm:injectivity-H0DB}.

\begin{intro-conjecture}\label{conj:injectivity-H0DB}
    Let $X$ be a variety with pre-$(k-1)$-Du Bois singularities. Then the map
    \[\RHom_{\O_X}(\DB_X^k,\omega_X^{\bullet})\to \RHom_{\O_X}(\H^0\DB_X^k,\omega_X^{\bullet})\]
    is injective on cohomology. 
\end{intro-conjecture}

This is proven in \cite{MP-lci} in the local complete intersection case, when $X$ has $(k-1)$-Du Bois singularities. The proof makes use however of the relationship
between the Hodge filtration and the Ext filtration on local cohomology, shown in \cite{MP-LC}, which is not available in general.

The natural extension of Theorem \ref{thm:vanishing-isolated} is the statement below. It follows from Conjecture \ref{conj:injectivity-H0DB}
with the same argument that derives Theorem \ref{thm:vanishing-isolated} from Theorem \ref{thm:injectivity-H0DB}.

\begin{intro-conjecture}\label{conj:vanishing-general}
 Let $X$ be a variety that is pre-$(k-1)$-Du Bois, and pre-$k$-Du Bois away from a closed subset of dimension $s$. Then
 \[\H^i\DB^k_X = 0 \quad \text{ {\rm for} } \quad 0<i<\depth\H^0\DB^k_X- s -1.\] 
\end{intro-conjecture}

\subsection*{Intersection complex}
In Ch.\ref{ch:IC} we establish analogues of the results described above, where the Du Bois complexes are replaced by \emph{intersection 
Du Bois complexes};\footnote{These are studied more thoroughly in the upcoming \cite{PP}.} in other words, from the point of view of Hodge modules and constructible sheaves, we are replacing the constant sheaf by the intersection complex. In this case, the higher Du Bois singularities conditions are replaced by higher rational singularities analogues; for local complete intersections, the link between these types of singularities and intersection complexes was already observed in \cite{CDM}.  Since their shape is rather similar, we refer to Sections \ref{scn:injectivity-IC} and \ref{scn:vanishing-IC} in the body of the paper for these statements.  The main injectivity result is Corollary \ref{cor:conj-IC}, while the main vanishing result is Corollary \ref{cor:IC-van}.

What is perhaps more fundamental here is that along the way we establish an injectivity result, Theorem \ref{thm:IC1}, that holds unconditionally and relates the duals of Du Bois complexes and their intersection analogues. This in turn uses a more technical variant, Theorem \ref{thm:injectivity-Rk}, communicated to us by S. G. Park.  Using this theorem, the main result in the intersection complex setting is a consequence of Theorem \ref{thm:injectivity-H0DB} for Du Bois complexes.

\medskip

\noindent 
\textbf{Acknowledgement.} We thank Brad Dirks, Mircea Musta\c t\u a, and especially Sung Gi Park, for valuable conversations and suggestions.

\section{Preliminaries}

Throughout this section, $X$ is a complex variety of dimension $n$.

\subsection{Du Bois complexes}

We recall the notion of \emph{filtered de Rham complex}, meant as a replacement for the standard de Rham complex on smooth varieties. Denoted 
$(\DB_X^\bullet, F)$, it is an object in the bounded derived category of filtered differential complexes on $X$, introduced by Du Bois in \cite{DB} along the lines suggested by work of Deligne.   For each $k \ge 0$, the shifted associated graded quotient
\[\DB^k_X : = \Gr^k_F \DB^\bullet_X[k],\]
is an object in $\DD^b_{\rm coh}(X)$, called the \emph{$k$-th Du Bois complex} of $X$. For a hyperresolution $\epsilon_{\bullet} \colon X_{\bullet}\to X$ of $X$, it can be computed as
\[\DB_X^k \simeq \bR \epsilon_{\bullet *} \Omega_{X_{\bullet}}^k.\]

Besides \cite{DB}, one can also find a detailed treatment of hyperresolutions and the construction of Du Bois complexes in \cite[Chapter V]{GNPP} or \cite[Chapter 7.3]{PS}. We only recall here a few basic facts that will be used freely throughout the paper:

\noindent
	$\bullet$ ~~For each $k\ge 0$, there is a canonical morphism $\Omega_X^k \to \DB_X^k$, which is an isomorphism if $X$ is smooth; here 
	$\Omega_X^k$ are the sheaves of K\"ahler differentials on $X$; see \cite[Section 4.1]{DB} or \cite[Page 175]{PS}. In particular, $\H^i\DB_X^k$ are supported on the singular locus of $X$, for all $i>0$.
	
\noindent
$\bullet$~~	There exists a Hodge-to-de Rham spectral sequence 
	$$E^{p,q}_1 = \HH^q (X, \DB_X^p) \implies H^{p + q} (X, \CC),$$
	which degenerates at $E_1$ if $X$ is projective; see \cite[Theorem 4.5(iii)]{DB} or \cite[Proposition 7.24]{PS}.

\noindent
$\bullet$~~For each $k\ge 0$, the sheaf $\H^0\DB_X^k$ embeds in $f_* \Omega_{\widetilde{X}}^k$, where $f\colon \widetilde{X} \to X$ is a resolution of $X$, 
so in particular it is torsion-free; see \cite[Remark 3.8]{h-differential}.

\subsection{Higher singularities}\label{sec:higher-singularities}
Following \cite{MOPW,JKSY,FL}, if $X$ is a local complete intersection (lci) subvariety of a smooth variety $Y$, then it is said to have \textit{$k$-Du Bois singularities} if the canonical morphisms $\Omega^p_X \to \DB^p_X$ are isomorphisms for all $0\le p\le k$, and 
 \textit{$k$-rational singularities} if the canonical morphisms $\Omega^p_X \to \DD_X(\DB^{n-p}_X)$  are isomorphisms for all $0\le p\le k$, where 
 $\DD_X (\cdot) : = \derR\H om (\cdot, \omega_X)$.
 
For non-lci varieties, however, even the condition $\Omega_X^1\xto{\sim}\DB_X^1$ turns out to be quite restrictive; as explained in \cite{SVV} the definitions above are not suitable anymore. 
In the general setting, new definitions of $k$-Du Bois and $k$-rational singularities are introduced in \textit{loc. cit.}.
As it is often sufficient, one can first consider weaker notions obtained by removing the conditions in cohomological degree 0.

\begin{definition}\label{definition:pre-k-DB-rational}
	We say that $X$ has \textit{pre-$k$-Du Bois} singularities if 
		\[ \H^i\DB_X^p=0 \,\,\,\,{\rm for~ all } \,\,\,\,i>0 \,\,\,\, {\rm and} \,\,\,\, 0\le p\le k.\]
We say that $X$ has \textit{pre-$k$-rational} singularities if
		\[ \H^i (\DD_X(\DB^{n-p}_X))=0 \,\,\,\,{\rm for~ all } \,\,\,\,i>0 \,\,\,\, {\rm and} \,\,\,\, 0\le p\le k.\]
\end{definition} 

Several other conditions are imposed in the full definition of general $k$-Du Bois and $k$-rational singularities. They agree with the classical notions of Du Bois and rational singularities when $k=0$, and  with the definitions mentioned above in the local complete intersection case. See \cite[Proposition 5.5, 5.6]{SVV} for more details.

\begin{definition} 
	We say that $X$ has \textit{$k$-Du Bois singularities} if it is seminormal, and 
	\begin{enumerate}
		\item $\codim_X (X_{\sing}) \ge 2k+1$;
		\item $X$ has pre-$k$-Du Bois singularities;
		\item $\H^0\DB_X^p$ is reflexive, for all $p \le k$.
	\end{enumerate}
\end{definition}

\begin{definition}
	We say that $X$ has \textit{$k$-rational singularities} if it is normal, and
	\begin{enumerate}
		\item $\codim_X (X_{\sing}) > 2k+1$;
		\item $X$ has pre-$k$-rational singularities. 
	\end{enumerate}
\end{definition}

\subsection{A new vanishing result}
Steenbrink's vanishing theorem \cite[(4.1)]{Steenbrink-vanishing} states that
\begin{equation}\label{eqn:Nakano}
\H^q\DB_X^p=0 \text{ for } p+q> n.
\end{equation}
In general, this result is the best possible. Indeed, \cite[Example 1.7]{MOPW} shows that the vanishing does not necessarily hold when $p+q=n$. However, we have

\begin{proposition}[Proposition \ref{prop:borderline-vanishing}]
If $k < n$ and $\H^{n-p -1}\DB_X^p=0$ for all $p \le k -1$, then 
\[\H^{n-k}\DB_X^k=0.\]
\end{proposition}
\begin{proof}
Consider the spectral sequence associated to the Hodge filtration on the filtered de Rham complex 
 \[E_1^{p,q} :=\H^q\DB_X^p \implies \H^{p+q}\DB_X^{\bullet}.\]
Since $\DB_X^{\bullet}$ is quasi-isomorphic to $\CC_X$, the spectral sequence converges to $\CC_X$, placed in cohomological degree $0$.
Note that for any $\ell \ge 1$, the term $E_{\ell+1}^{k,n-k}$ is obtained as the cohomology of the complex
$$E_{\ell}^{k- \ell,n-k + \ell - 1}\to E_\ell^{k,n-k}\to E_\ell^{k+ \ell,n-k - \ell + 1},$$
and the right hand side is $0$ by ($\ref{eqn:Nakano}$), while the left hand side is $0$ by assumption. Therefore
$$\H^{n-k}\DB_X^k = E_1^{k,n-k}=E_{\infty}^{k,n-k}=0.$$ 
\end{proof}

\begin{corollary}\label{cor:borderline-vanishing}
If $X$ is pre-$(k-1)$-Du Bois, with $k < n$, then
\[\H^{n-k}\DB_X^k=0.\]
In particular, $X$ is pre-$k$-Du Bois for all $k$ if and only if it is pre-$(n-2)$-Du Bois.
\end{corollary}
\begin{proof}
The first part is an immediate consequence of Proposition \ref{prop:borderline-vanishing}. For the second part, note that we always have 
$\DB_X^n \simeq \H^{0}\DB_X^n$, while in the previous row, the only term that needs checking is $\H^{1}\DB_X^{n-1}$, which is covered
 by the first part.
\end{proof}

When $X$ has isolated singularities, this is \cite[Lemma 2.5]{FL-isolated}. When $X$ is a hypersurface, a slightly weaker result is contained in \cite[Theorem 1.4]{MOPW}.

\begin{remark}
Note that the result does not hold for $k = n$, as for any variety $X$ we have 
$\DB_X^n \simeq \H^{0}\DB_X^n \simeq \pi_* \omega_{\widetilde X}$, where 
$\pi \colon \widetilde X \to X$ is a resolution of singularities.
\end{remark}

\begin{example}
If $X$ is a pre-$0$-Du Bois surface, then $\H^i \DB_X^k = 0$ for all $k$ and all $i > 0$. Hence if $X$ is a surface with rational 
singularities, then 
$$\DB_X^k  \simeq \H^0 \DB_X^k\simeq \Omega_X^{[k]} \,\,\,\,\,\,{\rm for~all~}k.$$
The last isomorphism is a consequence of the main result of \cite{KS-extending}.
\end{example}

\begin{example}
Corollary \ref{cor:borderline-vanishing} cannot be improved, without further assumptions, by moving to the left in the Du Bois table. For instance, 
any $3$-fold $X$ with an isolated rational (hence Du Bois) hypersurface singularity that is not a double point, has $\H^{1}\DB_X^1 \neq 0$; see 
\cite[Theorem 2.2]{NS}. 
\end{example}

\subsection{Local cohomological dimension and depth}\label{scn:lcd}
Let $X$ be a complex variety. If $Y$ is a smooth variety containing $X$ (locally), the local cohomological dimension of $X$ in $Y$ is
$$\lcd (X, Y) := \max ~\{q ~|~ \H^q_X \O_Y \neq 0\},$$
where the sheaf in the parenthesis is the $q$-th local cohomology sheaf of $\O_Y$ along $X$. It is also known that if $r = \codim_Y X$, then $\H^q_X \O_Y = 0$ for $q < r$ and $\H^r_X \O_Y \neq 0$; 
see e.g. \cite{MP-LC} for a more details and references.

As in \cite{PSh}, we consider the \emph{local cohomological defect} ${\rm lcdef} (X)$ of $X$ as
$${\rm lcdef} (X) : = \lcd (X, Y) - \codim_Y X.$$
A reinterpretation of the characterization of local cohomological dimension in \cite[Theorem E]{MP-LC} can be stated as follows:

\begin{theorem}[{\cite[Corollary 12.6]{MP-LC}}]\label{thm:LCD}
Let $X$ be a subvariety of a smooth variety $Y$. For any integer $c$ we have 
$$\lcd(X, Y) \le c  \iff \sExt^{j + k + 1}_{\O_Y}(\DB_X^k,\omega_Y) = 0 \,\,\,\,{\rm for ~all~}j \ge c \,\,\,\,{\rm and}\,\,\,\, k \ge 0.$$
or equivalently 
$${\rm lcdef} (X) \le c  \iff  \sExt^{j + k + 1}_{\O_X}(\DB_X^k,\omega_X^{\bullet}) = 0 \,\,\,\,{\rm for ~all~}j \ge c - \dim X \,\,\,\,{\rm and}\,\,\,\, k \ge 0.$$
\end{theorem}

The second equivalence follows from the first thanks to Grothendieck duality for the inclusion $X \hookrightarrow Y$.  

We recall now that the notion of depth of a module has a natural extension to obejcts in the derived category. If $C$ is an element of the bounded derived category of finitely generated $R$-modules, where $(R, \frak{m})$ be a noetherian local ring endowed with a dualizing complex 
$\omega_R^\bullet$, then one can define
$$\depth(C) : = \min~  \{i ~|~\Ext^{-i}_R (C,\omega_R^{\bullet})\neq 0\}.$$
with the convention that the depth is $-\infty$ if $C = 0$. This notion is studied extensively in \cite{FY}, where it is shown to be equivalent
to other natural generalizations of the usual notion of depth. 
When $X$ is a variety and $\C$ is an element in $D^b_{\rm coh} (X)$, then we set 
$$\depth(\C) : = \min_{x \in {\rm Supp} (\C)}~ \depth(\C_x),$$
where the minimum is taken over the closed points in the support of $\C$. The first interpretation takes the form: 
\begin{equation}\label{eqn:depth}
\depth(\C) = \min~  \{i ~|~\sExt^{-i}_{\O_X} (\C,\omega_X^{\bullet})\neq 0\}.
\end{equation}
This is of course a standard interpretation of depth when $\C$ is a sheaf.

Using this, from Theorem \ref{thm:LCD} we conclude:

\begin{corollary}\label{cor:depth-DB-complex}
We have the identity 
$${\rm lcdef} (X) = \dim X -  \underset{k \ge 0}{\min}~ \{ \depth \DB_X^k + k \}.$$
\end{corollary}

This shows in particular that ${\rm lcdef} (X)$ depends only on $X$, and not on the embedding, and that 
 $\dim X \ge {\rm lcdef} (X) \ge 0$.

\begin{example}[{\bf Varieties with ${\rm lcdef}(X)= 0$}]\label{def=0}
The condition ${\rm lcdef}(X)= 0$ is equivalent to $\lcd (X, Y) = \codim_Y X$ in any embedding, or equivalently to the non-vanishing of a single 
local cohomology sheaf $\H^r_X \O_Y$. This holds of course when $X$ is a local complete intersection. In addition, it is known to hold 
when $X$ has quotient singularities \cite[Corollary 11.22]{MP-LC}, for affine varieties with Stanley-Reisner coordinate algebras that are Cohen-Macaulay \cite[Corollary 11.26]{MP-LC}, for arbitrary Cohen-Macaulay surfaces and threefolds in \cite[Remark p.338-339]{Ogus} and 
\cite[Corollary 2.8]{DT} respectively, and for Cohen-Macaulay fourfolds whose local analytic Picard groups are torsion \cite[Theorem 1.3]{DT}.  

Note that, according to Corollary \ref{cor:depth-DB-complex}, for such varieties we have
$$ \depth \DB_X^k \ge n- k \,\,\,\,\,\,{\rm for ~all}\,\,\,\, k \ge 0.$$
\end{example}

\smallskip

We finish  by recalling a well-known vanishing result for Ext sheaves, for later use:

\begin{lemma}\label{lem:vanishing-of-Ext}
Let $\F$ be a coherent sheaf on $X$. Then 
\[\sExt^i_{\O_X}(\F, \omega_X^{\bullet})=0 \text{ \,\,{\rm for} \,\,} i< -\dim {\rm Supp}( \F).\]
\end{lemma}

\section{Injectivity theorems}
In this chapter we address natural injectivity theorems for the cohomologies of the duals of the graded quotients of the Du Bois complex. The  first place where such a result appears is the paper \cite{KS-injectivity} by Kov\'acs-Schwede, who proved that for every variety $X$ the morphism 
\[\RHom_{\O_X} (\DB_X^0,\omega_X^{\bullet})\to \RHom_{\O_X} (\O_X,\omega_X^{\bullet})\]
obtained by dualizing the canonical morphism $\O_X \to \DB_X^0$,  is injective on cohomology; see \cite[Theorem 3.3]{KS-injectivity}. Using the Hodge filtration on local cohomology, a slightly stronger version of this fact was obtained in \cite[Theorem A]{MP-LC}, and then extended to higher Du Bois complexes in \cite[Theorem A]{MP-lci} in the case of local complete intersections:

\begin{theorem}[{\cite[Theorem A]{MP-lci}}]\label{thm:MP-injectivity}
If $X$ is local complete intersection, and $k$ is a nonnegative integer such that $X$ has $(k - 1)$-Du Bois singularities, then the morphism
\[\RHom_{\O_X} (\DB_X^k,\omega_X)\to \RHom_{\O_X} (\Omega_X^k,\omega_X)\]
in $\DD^b_{\rm coh} (X)$, obtained by dualizing the canonical
morphism $\Omega_X^k \to\DB_X^k$, is injective at the level of cohomology.
\end{theorem}

Here we prove an injectivity theorem for arbitrary isolated singularities, by going back to the basic Hodge-theoretic properties of Du Bois complexes, as in \cite{KS-injectivity}. Currently we do not have a sufficiently good understanding of the Hodge filtration on local cohomology, 
as a mixed Hodge module, beyond the local complete intersection case treated in \cite{MP-LC}. This is something highly desirable, that may clarify
the picture in the general non-isolated case and lead to a proof of Conjecture \ref{conj:injectivity-H0DB}.

\subsection{Proof of Theorem \ref{thm:injectivity-H0DB}}\label{scn:inj}
The statement of the injectivity theorem is local, hence we may assume first that $X$ is quasi-projective. Since the singular locus $S$ of $X$ is 
a finite set, we may choose a compactification $\bar{X}$  of $X$ such that the singular locus of $\bar{X}$ is still $S$, and prove the statement for $\bar{X}$.
Hence it suffices to assume that $X$ is projective to begin with. With this assumption, we prove a stronger statement:

\begin{theorem}\label{thm:injectivity-H0DB-gen}
    Let $X$ be a projective variety which is pre-$(k-1)$-Du Bois, and pre-$k$-Du Bois away from a finite set. Then the morphism
    \[\RHom_{\O_X}(\DB_X^k,\omega_X^{\bullet})\to \RHom_{\O_X}(\H^0\DB_X^k,\omega_X^{\bullet})\]
    in $\DD^b_{\rm coh} (X)$, obtained by dualizing the canonical map $\H^0\DB_X^k \to \DB_X^k$,
    is injective on cohomology. 
\end{theorem}

The key point in the proof is the following:

\begin{proposition}\label{surjection}
Let $X$ be a projective variety with pre-$(k-1)$-Du Bois singularities. Then for each $i$, the natural map
$$H^i (X, \H^0 \DB_X^k) \to \HH^i (X, \DB_X^k),$$
obtained by applying cohomology to $\H^0\DB_X^k \to \DB_X^k$, is surjective.
\end{proposition}
\begin{proof}
For each $p \ge 0$, we denote 
$$\DB_X^{\le p} : = \DB_X^\bullet / F^{p + 1} \DB_X^\bullet.$$
So we have an exact triangle 
\begin{equation}\label{eq:tr1}
\DB^p_X[-p] \longrightarrow \DB_X^{\le p} \longrightarrow  \DB_X^{\le p-1} \overset{+1}\longrightarrow.
\end{equation}
We also denote by $\Omega_{X,h}^{\le p}$ the object in the derived category of differential complexes on $X$\footnote{The notation is motivated by the fact that 
	$\H^0\DB_X^k$ agrees with the $h$-differentials $\Omega_{X,h}^k$ studied in \cite{h-differential}.}, represented by the complex 
\[[\H^0\DB^0_X \xto{d} \H^0\DB^1_X \xto{d} \cdots \xto{d} \H^0\DB^p_X],\]
placed in cohomological degrees $0,\dots, p$. This is not to be confused with $\H^0 (\DB_X^{\le p})$. Here we have an exact triangle
\begin{equation}\label{eq:tr2}
\H^0 \DB^p_X[-p] \longrightarrow \Omega_{X,h}^{\le p} \longrightarrow \Omega_{X,h}^{\le p-1} \overset{+1}\longrightarrow.
\end{equation}

Note now that, as in \cite[Proposition 2.3]{SVV}, there exists a natural map 
$\Omega_{X,h}^{\le p}\to \DB_X^{\le p}$. When $X$ is projective, the $E_1$-degeneration of the Hodge-to-de Rham spectral sequence for the filtered de Rham complex of $X$ implies that the induced composition
    \[H^i(X, \CC) \to \HH^i(X, \Omega_{X,h}^{\le p}) \to \HH^i(X, \DB^{\le p}_X)\]
 is surjective for each $i$, hence so is the second map. 

Let's now consider the integer $k$ in the statement.
The map $\Omega_{X,h}^{\le k}\to \DB_X^{\le k}$, and its analogue for $k-1$, combined with the two exact triangles described above, gives rise 
to a morphism of exact triangles
\[\begin{tikzcd}
	\H^0\DB^k_X[-k] & \Omega_{X,h}^{\le k} & \Omega_{X,h}^{\le k-1}  \overset{+1}\longrightarrow \\
	\DB^k_X[-k] & \DB_X^{\le k} & \DB_X^{\le k-1}   \overset{+1}\longrightarrow.
         \arrow[from=1-1, to=1-2]
	\arrow[from=1-2, to=1-3]
	\arrow[from=2-1, to=2-2]
	\arrow[from=2-2, to=2-3]
	\arrow[from=1-1, to=2-1]
	\arrow[from=1-2, to=2-2]
	\arrow[from=1-3, to=2-3]
    \end{tikzcd}\]
Since $X$ is pre-$(k-1)$-Du Bois, the right-most vertical map is an isomorphism.  Passing to hypercohomology, we obtain a morphism of long exact sequences

 \[\begin{tikzcd}
	H^i (\H^0 \DB_X^k [-k]) & \HH^i (\Omega_{X, h}^{\le k})&  \HH^i (\Omega_{X, h}^{\le k-1})& H^{i+ 1} (\H^0 \DB_X^k [-k])  \\
	 \HH^i (\DB_X^k [-k]) &  \HH^i (\DB_X^{\le k}) &   \HH^i (\DB_X^{\le k-1})  &  \HH^{i+1} (\DB_X^k [-k]) 
	\arrow[from=1-1, to=1-2]
	\arrow[from=1-2, to=1-3]
	\arrow[from=1-3, to=1-4]
	\arrow[from=2-1, to=2-2]
	\arrow[from=2-2, to=2-3]
	\arrow[from=1-1, to=2-1]
	\arrow[from=1-2, to=2-2]
	\arrow["\simeq", from=1-3, to=2-3]
	\arrow[from=2-3, to=2-4]
	\arrow[from=1-4, to=2-4]
    \end{tikzcd}\]
Since the second vertical map is surjective for all $i$, basic homological algebra shows that so is the first. 
\end{proof}

We now consider the exact triangle
$$\H^0 \DB_X^k \to \DB_X^k \to  C \overset{+1}{\longrightarrow}.$$
By definition $X$ is pre-$k$-Du Bois away from a finite set of points if and only if $C$ is supported on a finite set. 
After dualizing, we obtain an exact triangle 
$$K \to  \RHom_{\O_X}(\DB_X^k,\omega_X^{\bullet}) \to \RHom_{\O_X}(\H^0\DB_X^k,\omega_X^{\bullet}) \overset{+1}{\longrightarrow},$$
where again $K$ is supported on a finite set.
Applying Grothendieck-Serre duality to the surjections in Proposition \ref{surjection}, we obtain that the induced morphisms 
$$\HH^i \big(X, \RHom_{\O_X}(\DB_X^k,\omega_X^{\bullet})\big)\to \HH^i \big(X, \RHom_{\O_X}(\H^0\DB_X^k,\omega_X^{\bullet})\big)$$
are injective for all integers $i$.

Theorem \ref{thm:injectivity-H0DB-gen} is then a consequence of the following general result:

\begin{lemma}
Let $X$ be a projective variety, and let
$$K \to F \to G \overset{+1}{\longrightarrow}$$ 
be an exact triangle in $\DD^b_{\rm coh} (X)$. Suppose that $K$ has zero-dimensional support, and that the induced maps on 
hypercohomology 
$$\HH^i (X, F) \to \HH^i (X, G)$$
are injective for all i. Then the induced maps on cohomology 
$$\H^i F \to \H^i G$$
are injective for all $i$.
\end{lemma}
\begin{proof}
The injectivity on hypercohomology implies that for each $i$ we have short exact sequences:
$$0 \to \HH^i (X, F) \to \HH^i (X, G) \to \HH^{i +1} (X, K) \to 0.$$
Now the hypercohomology of $G$ is computed by a spectral sequence
$$E^{p,q}_2 = H^p(X, \H^q G) \implies \HH^{p+q}(X, G),$$
while the similar spectral sequence for $K$ shows that 
$$\HH^{i+1} (X, K) \simeq H^0 (X, \H^{i +1} K),$$
because of the assumption that $K$ is supported in dimension zero. Passing to the first associated graded term of the filtration on the total object 
in each of these two cases leads to a commutative diagram
   \[\begin{tikzcd}
        &\HH^i(X, G)\ar[d] \ar[r]& \HH^{i + 1} (X, K)\ar[d] \\
        &E^{0, i}_\infty \ar[r]& H^0 (X, \H^{i +1} K)
    \end{tikzcd}\]
and by the observations above, it follows that the bottom horizontal map is surjective. On the other hand, note that in fact this map has a factorization
$$E^{0, i}_\infty \hookrightarrow E^{0, i}_2 = H^0 (X, \H^i G) \overset{\varphi}{\longrightarrow}  H^0 (X, \H^{i +1}  K),$$
where $\varphi$ comes from the connecting homomorphism $H^i G \to \H^{i +1} K$ induced by the original triangle. Since the support of $\H^{i+1} K$ is zero-dimensional, it follows immediately that this connecting homomorphism is surjective for each $i$, which is equivalent to our assertion.
\end{proof}

\subsection{Injectivity results involving K\"ahler differentials}
Recall that for any $k \ge 0$ we have natural maps 
$$\Omega_X^k \longrightarrow \H^0\DB_X^k \longrightarrow \DB_X^k.$$
Since $\H^0\DB_X^k$ is known to be torsion-free, this arises in fact from a sequence of maps 
 \begin{equation}\label{eqn:composition-tf}
 \Omega_{X, {\rm tf}}^k \overset{\alpha}{\longrightarrow} \H^0\DB_X^k \overset{\beta}{\longrightarrow} \DB_X^k,
 \end{equation}
 where $\Omega_{X, {\rm tf}}^k: = \Omega_X^k / {\rm tors} \big( \Omega_X^k \big)$ is the canonical torsion-free quotient of the sheaf of K\"ahler differentials, and $\alpha$ is an inclusion which is an isomorphism away from $X_{\sing}$. Our main injectivity theorem has the following consequence:

\begin{corollary}\label{cor:injectivity-MP-gen}
Let $X$ be a variety with isolated pre-$(k-1)$-Du Bois singularities.  Then the dual 
        \[\RHom_{\O_X}(\DB_X^k,\omega_X^{\bullet})\to \RHom_{\O_X}(\Omega_{X, {\rm tf}}^k,\omega_X^{\bullet})\]
of the canonical morphism 
$(\ref{eqn:composition-tf})$  is injective on the $i$-th cohomology for all $i \neq 0$. Moreover, if  ${\rm lcdef} (X) \le \dim X - k - 1$,  then it is injective on all cohomologies.
\end{corollary}
    \begin{proof}
Given Theorem \ref{thm:injectivity-H0DB},  for the first statement it is enough to have the injectivity on cohomology of the map
  \[\RHom_{\O_X}(\H^0\DB_X^k,\omega_X^{\bullet})\to \RHom_{\O_X}(\Omega_{X, {\rm tf}}^k,\omega_X^{\bullet})\]
  obtained by dualizing $\alpha$, in the range $i\neq 0$. But this is clear; if we complete $\alpha$ to a short exact sequence
    \[0\longrightarrow \Omega_{X, {\rm tf}}^k\longrightarrow \H^0\DB_X^k\longrightarrow Q\longrightarrow 0,\]
the cokernel $Q$ is supported on $X_{\sing}$, hence Lemma \ref{lem:vanishing-of-Ext} and ($\ref{eqn:depth}$)  imply $\sExt^i(Q,\omega_X^{\bullet})=0$ for $i\neq 0$. 
Dualizing the short exact sequence above then shows that in this range
        $$\sExt_{\O_X}^i(\H^0\DB_X^k,\omega_X^{\bullet})\to \sExt_{\O_X}^i(\Omega_{X, {\rm tf}}^k,\omega_X^{\bullet})$$
        is injective.
For the second statement, simply note that Theorem \ref{thm:LCD} implies that 
$$\sExt_{\O_X}^i(\DB_X^k,\omega_X^{\bullet})=0 \,\,\,\,\,\,{\rm for}\,\,\,\, i> {\rm lcdef} (X) -\dim X+k,$$
so that the assumption takes care of the remaining case $i = 0$.
            \end{proof}
      
\medskip            

We compare this with the previous injectivity statements obtained in the literature. 

\noindent
$\bullet$~When $k=0$,  we recall that the injectivity on cohomology of the canonical morphism 
\[\RHom_{\O_X}(\DB_X^0,\omega_X^{\bullet})\to \RHom(\O_{X},\omega_X^{\bullet})\]
holds in full generality thanks to \cite{KS-injectivity}.

\noindent
$\bullet$~When $k\ge 1$, for isolated singularities we obtain the extension of a
 result shown in the local complete intersection case in \cite{MP-lci}. Note first that by definition, when $X$ is $(k-1)$-Du Bois, rather than just pre-$(k-1)$-Du Bois, we have $\codim X_{\sing}\ge 2k-1$; when $X$ is a local complete intersection, this is not part of the definition, but holds automatically by \cite[Theorem F and Corollary 9.26]{MP-LC}. In our case this simply means $n - k - 1 \ge k - 2$, hence as a consequence of Corollary \ref{cor:injectivity-MP-gen} we first obtain:      
            
\begin{corollary}\label{cor:k-large}
When $X$ has isolated $(k-1)$-Du Bois singularities, and ${\rm lcdef  (X)} \le k - 2$, the map in Corollary \ref{cor:injectivity-MP-gen} is injective on all cohomologies. 
\end{corollary}            

We deduce the promised analogue of the result in \cite{MP-lci}:

\begin{corollary}\label{cor:injectivity-MP}
 Let $X$ be a variety with isolated $(k-1)$-Du Bois singularities. If ${\rm lcdef} (X) = 0$, then the map
        \[\RHom_{\O_X}(\DB_X^k,\omega_X^{\bullet})\to \RHom_{\O_X}(\Omega_{X, {\rm tf}}^k,\omega_X^{\bullet})\]
 is injective on cohomology. 
 \end{corollary}   
    \begin{proof}
 Note that for $k = 0$ the result holds with no assumptions, so we may assume $k \ge 1$.     
 When $k \ge 2$, the result is a consequence of Corollary \ref{cor:k-large}. When $k =1$ and $n\ge 2$, we go back directly to the statement of 
 Corollary  \ref{cor:injectivity-MP-gen}. When $k=n=1$, since $\depth \H^0\DB_X^1=\depth \Omega^1_{X,{\rm tf}} = 1$, the map
 \[\RHom_{\O_X}(\DB_X^k,\omega_X^{\bullet})\to \RHom_{\O_X}(\Omega_{X,{\rm tf}}^k,\omega_X^{\bullet})\]
 is nontrivial on the $i$-th cohomology only when $i=-1$, in which case its injectivity follows from Corollary \ref{cor:injectivity-MP-gen}. 
    \end{proof}

This applies in particular when $X$ is a local complete intersection. When $\dim X\ge 2$ or $k\ge 2$, under our assumptions \cite[Corollary 3.1]{MV} implies that $\Omega_X^k$ is torsion-free, hence the injectivity on cohomology holds directly for      \[\RHom_{\O_X}(\DB_X^k,\omega_X^{\bullet})\to \RHom_{\O_X}(\Omega_X^k,\omega_X^{\bullet}).\footnote{In fact this can be easily seen to hold when $\dim X = k = 1$ as well.}\]
However, other interesting classes of varieties satisfy ${\rm lcdef} (X) = 0$ as well, and Corollary \ref{cor:injectivity-MP} also holds for those;  see Example \ref{def=0}.

\begin{remark}[{\bf Non-isolated singularities}]
Assuming Conjecture \ref{conj:injectivity-H0DB} (or under the hypothesis of Theorem \ref{thm:injectivity-H0DB-gen}), one has analogues of the results in this section  for projective varieties with possibly non-isolated singularities. The conclusion of Corollary \ref{cor:injectivity-MP-gen} becomes the fact that injectivity holds on $i$-th cohomology for:
\begin{enumerate}
\item $i > \lcdef (X) - \dim X + k$; in this case in fact $\sExt_{\O_X}^i(\DB_X^k,\omega_X^{\bullet})=0$.
\item $i < - \dim X_{\sing}$.
\end{enumerate}
\end{remark}


\section{Applications of injectivity}

\subsection{Vanishing of higher cohomology}\label{scn:vanishing}
In this section we prove Theorem \ref{thm:vanishing-isolated} and explain some related points.

We first state a general homological result about the vanishing of cohomologies of objects in the derived category of coherent sheaves, in terms of their depth.

\begin{proposition}\label{prop:derived}
Let $A^\bullet$ be an object in $\DD^b_{\rm coh} (X)$ such that 
\begin{enumerate}
\item $A^\bullet$ has cohomology in non-negative degrees.
\item The support of all $\H^i A^\bullet$ with $i > 0$ is contained in a closed subset of $X$ of dimension $s$.
\end{enumerate}
Then we have
\[\H^i A^\bullet  = 0 \quad \text{ {\rm for} } \quad 0<i< {\rm min} \{\depth \H^0 A^\bullet, \depth A^\bullet + 1\} - s -1.\]
\end{proposition}
\begin{proof}
The first assumption implies that there is a natural morphism $\H^0 A^\bullet  \to A^\bullet$, and taking its Grothendieck dual 
leads to
\begin{equation} \label{eqn:cone of non-DB locus}
\RHom_{\O_X}(A^\bullet,\omega_X^{\bullet}) \xto{\varphi} \RHom_{\O_X}(\H^0 A^\bullet,\omega_X^{\bullet}) \to C^{\bullet}\xto{+1}, 
\end{equation}
where $C^{\bullet}$ is the cone of the morphism $\varphi$. For each $q$, we obtain exact sequences
\[\sExt^q_{\O_X}(\H^0 A^\bullet,\omega_X^{\bullet})\to \H^q C^{\bullet}\to  \sExt^{q+1}_{\O_X}(A^\bullet,\omega_X^{\bullet}). \]
Using ($\ref{eqn:depth}$), the first term vanishes for $q>-\depth \H^0 A^\bullet$, and the third for  $q> -\depth A^\bullet - 1$. We conclude that 
\[\H^q C^{\bullet}=0\text{ for } q> - m,\]
where $m : = {\rm min} \{\depth \H^0 A^\bullet, \depth A^\bullet\}$.

Next, applying $\RHom_{\O_X}(-, \omega_X^{\bullet})$ to (\ref{eqn:cone of non-DB locus}), we obtain an exact triangle 
\[\RHom_{\O_X}(C^{\bullet},\omega_X^\bullet)\to \H^0 A^\bullet  \to A^\bullet \xto{+1}.\]
It follows that for $i>0$,
\[\H^i A^\bullet \simeq \sExt^{i+1}_{\O_X}(C^\bullet, \omega_X^{\bullet}).\]
Consider now the spectral sequence 
\[E^{p,q}_2 = \sExt^p_{\O_X}(\H^q C^{\bullet},\omega_X^\bullet) \implies \sExt^{p-q}_{\O_X}(C^{\bullet},\omega_X^\bullet).\]
We have:
\begin{itemize}
    \item $\H^qC^{\bullet}=0$ for $q> - m$, as noted above;
    \item $\sExt^p_{\O_X}(\H^q C^{\bullet},\omega_X^{\bullet})=0$ for $p<-\dim \supp \H^q C^{\bullet}$, by Lemma \ref{lem:vanishing-of-Ext}. In particular, this holds for $p<-s$, since by definition $C^{\bullet}$ is supported on the locus where $\H^0 A^\bullet$ and $A^\bullet$ are not 
quasi-isomorphic, which has dimension $s$.
\end{itemize}
Combining these facts, we see that $\sExt^i_{\O_X}(C^{\bullet},\omega_X^{\bullet})=0$ for $i< m - s$. Thus
\[\H^i A^\bullet = 0 \,\,\,\,\,\,{\rm for} \,\,\,\,0<i< m - s -1.\]
\end{proof}

\begin{proof}[Proof of Proposition \ref{prop:vanishing-general}]
We simply take $A^\bullet = \DB_X^k$ in Proposition \ref{prop:derived}. Its higher cohomologies are supported on the 
non-pre-$k$-Du Bois locus of $X$, so we obtain
\[\H^i\DB^k_X = 0 \quad \text{ {\rm for} } \quad 0<i< m_k^\prime - s -1,\]
 where 
 $$m_k^\prime := {\rm min} \{\depth\H^0\DB^k_X, \depth \DB^k_X+ 1\}.$$
The result then follows from the characterization of the local cohomological defect in Corollary \ref{cor:depth-DB-complex}.
\end{proof}

\begin{proof}[Proof of Theorem \ref{thm:vanishing-isolated}]
To deduce the stronger vanishing statement in the case of isolated singularities, 
the key new ingredient is that, thanks to Theorem \ref{thm:injectivity-H0DB}, with the notation in (1) the long exact sequence on cohomology
associated to the triangle ($\ref{eqn:cone of non-DB locus}$) breaks into short exact sequences:
\[0\to \sExt^i_{\O_X}(A^\bullet,\omega_X^{\bullet})\to \sExt^i_{\O_X}(\H^0 A^\bullet ,\omega_X^{\bullet})\to \H^i C^{\bullet}\to 0 \]
for all $i$. 

The inclusion of Ext sheaves gives 
$$\depth \DB_X^k \ge  \depth \H^0\DB_X^k$$
on any variety with pre-$(k-1)$-Du Bois singularities. Hence in the proof of Proposition \ref{prop:vanishing-general} we 
have $m_k^\prime =  \depth\H^0\DB^k_X$, and $s = 0$, which gives the desired result.
\end{proof}

The exact same argument, using the injectivity theorem of Kov\'acs-Schwede \cite{KS-injectivity} in place of Theorem \ref{thm:injectivity-H0DB}, proves Corollary \ref{cor:CM}.

\begin{remark}\label{rmk:vanishing-log-poles}
	Let $X$ be a variety with isolated singular locus $S$, and let $f \colon \widetilde{X}\to X$ be a resolution of singularities with simple normal crossings exceptional divisor $E=f^{-1}(S)_{\rm red}$. The vanishing result for the Du Bois complexes of $X$ in Theorem \ref{thm:vanishing-isolated} can be reformulated as saying that
	\[R^if_*\Omega^k_{\widetilde{X}}(\log E)(-E) = 0\quad \text{ {\rm for} } \quad 0<i<\depth\H^0\DB^k_X -1,\]
	assuming that $X$ has pre-$(k-1)$-Du Bois singularities. 
	Indeed, by \cite[Proposition 3.3]{Steenbrink-vanishing} there is an exact triangle
	\[\bR f_*\Omega^k_{\widetilde{X}}(\log E)(-E)\to \DB_X^k\to \Omega_S^k\xto{+1}.\]
	Thus, for $k>0$, it is immediate that
	\[\H^i \DB_X^k \simeq R^{i}f_*\Omega^k_{\widetilde{X}}(\log E)(-E)\]
	for all $i$. 
	When $k=0$ and $i > 0$, this isomorphism still holds: the map $\H^0\DB_X^0\to \O_S$ is surjective, since the composition $\O_X \to \H^0 \DB_X^0 \to \O_S$ is the natural surjection. 
\end{remark}

\begin{example}[{\bf The LCI case}]\label{ex:vanishing-lci}
If $X$ is a local complete intersection with $\dim X_{\rm sing} = s$, it is shown in \cite[Corollary 13.9]{MP-LC}, using the Hodge filtration on local cohomology, that 
\begin{equation}\label{eqn:lci}
\H^i\DB_X^k = 0 \quad\text{ for }\quad 0<i< n-s-k-1.
\end{equation}
We now explain that this fact is a special case of our results here. The statement is vacuous when $\codim X_{\rm sing} = n - s \le k+2$, hence we may assume $n -s \ge k +3$, which in particular implies that $X$ is normal and moreover, by \cite[Corollary 3.1]{MV},  that the sheaf of K\"ahler differentials $\Omega_X^k$ is reflexive.
This in turn implies that $\H^0\DB_X^k = \Omega_X^{[k]}=\Omega_X^k$. By Lemma 1.8 of \cite{Greuel}, within this range we then have 
    \[\depth \H^0\DB_X^k=\depth \Omega_X^k\ge n -k.\]
Since ${\rm lcdef}(X) = 0$, ($\ref{eqn:lci}$) then follows from Proposition \ref{prop:vanishing-general}.
\end{example}

\noindent
{\bf A criterion for the Du Bois condition.}
A simple but intriguing consequence of vanishing in the form of Corollary \ref{cor:CM} is the following:

\begin{corollary}\label{cor:DB-CM-criterion}
Let $X$ be a projective seminormal Cohen-Macaulay variety of dimension $n$, with isolated singularities, or more generally 
Du Bois away from a finite set of points.
If $H^n (X, \O_X) = 0$, then $X$ is Du Bois.

More precisely, we have $h^n (X, \O_X) \ge h^n (X, \DB^0_X)$, and $X$ is Du Bois $\iff$ $h^n (X, \O_X) = h^n (X, \DB^0_X)$
$\iff$ the natural map $H^n (X, \O_X) \to \HH^n (X, \DB_X^0)$  is injective (hence an isomorphism).
\end{corollary}
\begin{proof}
We consider the cone 
$$\O_X \to \DB_X^0 \to C^\bullet \overset{+1}{\longrightarrow}.$$
Note that $C^\bullet$ is supported on a finite set.
We clearly  have $\H^i C^\bullet = 0$ for $i \le 0$ (since the seminormality condition is equivalent to $\O_X \simeq \H^0 \Omega_X^0$), while $\H^i C^\bullet \simeq \H^i \DB_X^0$ for $i \ge 1$. 
In particular, by Proposition \ref{prop:borderline-vanishing} we have $\H^i C^\bullet = 0$ for $i > n -1$. Moreover, since $X$ is Cohen-Macaulay, by 
Corollary \ref{cor:CM} we have $\H^i C^\bullet = 0$ for $i< n-1$.

Note now that we have a short exact sequence 
$$0 \to \HH^{n-1} (X, C^\bullet) \to H^n (X, \O_X) \to \HH^n (X, \DB_X^0) \to 0.$$
The last map is surjective thanks to the degeneration of the Hodge-to-de Rham spectral sequence, as in Section \ref{scn:inj}, as it sits in the surjective composition
$$H^n (X, \CC) \to H^n (X, \O_X) \to  \HH^n (X, \DB_X^0).$$

The hypercohomology group $\HH^{n-1} (X, C^\bullet)$ is computed by a spectral sequence whose $E_2$-terms are 
$$E^{p, q}_2 = H^p (X, \H^q C^\bullet), \,\,\,\,\,\,{\rm with}\,\,\,\,p +q = n-1.$$
Since $\H^i C^\bullet = 0$ for $i \neq n-1$, this gives
$$H^n (X, \O_X) \simeq \HH^n (X, \DB_X^0) \iff H^0 (X, \H^{n-1} C^\bullet) = 0.$$
As the support of $C^\bullet$ is finite, this last condition is equivalent to $\H^{n-1} C^\bullet = 0$, hence to $C^\bullet = 0$, i.e. to $X$ being Du Bois. \end{proof}

For instance this applies to any low degree normal complete intersection with isolated singularities in $\PP^n$.

\begin{remark}
The criterion above has a (rather technical) analogue for higher $k$: using Theorem \ref{thm:vanishing-isolated}, the same proof shows that if $X$ is pre-$(k-1)$-Du Bois, and 
pre-$k$-Du Bois away from a finite set, and if $\depth \H^0 \DB_X^k = n-k$, then the cohomology vanishing $H^{n-k} (X,  \H^0 \DB_X^k) = 0$ 
implies that $X$ is pre-$k$-Du Bois.
\end{remark}

\noindent
{\bf Examples of non-vanishing.}
In \cite[Question 13.10]{MP-LC} it is asked whether the vanishing result for local complete intersections in Example \ref{ex:vanishing-lci}, namely 
\[\H^i\DB_X^k = 0 \quad\text{ for }\quad 0<i< n-s-k-1\]
with $s=\codim X_{\sing}$, continues to hold when $X$ is arbitrary, or at least Cohen-Macaulay.

The study of Du Bois complexes of cones in \cite{SVV} and \cite{PSh} provides simple counterexamples. 

\begin{example}	
First a very simple example that is not Cohen-Macaulay. Let $X =C( Y,L)$ be the abstract affine cone over a smooth projective threefold $X$ endowed with an ample line bundle $L$ such that $H^1(Y,L)\neq 0$.\footnote{For example, we can take $X= C\times C\times C$ for some smooth projective curve $C$ of genus $g\ge 2$, and $L = \O_C(p)\boxtimes \O_C(p)\boxtimes \O_C (p)$ for some $p \in C$.}
Then according to \cite[Proposition 7.2]{SVV}, we have $\H^1\DB_X^0\neq 0$ and $\H^1\DB_X^1\neq 0$.
\end{example}

\begin{example}
In this example $X$ has rational, hence in particular Cohen-Macaulay, singularities. Let $Y$ be a smooth Fano threefold for which 
	\[H^1(Y, \Omega_Y^2\otimes L)\neq 0,\]
	where $L=\omega_Y^{-1}$. The existence of such $Y$ is shown in \cite[Section 2]{Totaro}.
	
Let now $X =C( Y\times Y, L\boxtimes L)$ be the abstract cone over $Y\times Y$ associated to the ample line bundle $L\times L$. Since 
$Y\times Y$ is still a Fano variety, we have 
$$H^i (Y \times Y, (L \boxtimes L)^m) = 0 \,\,\,\,\,\,{\rm for~all} \,\,\,\,i > 0, m\ge 0,$$
hence $X$ has rational singularities; see e.g. \cite[Remark 7.8]{SVV}. Furthermore, we have 
	\[H^1(Y\times Y, \Omega_{Y\times Y}^2\otimes (L\boxtimes L))\neq 0,\]
	as it contains $H^1(Y, \Omega_Y^2\otimes L)\neq 0$ as a direct summand. Using \cite[Proposition 7.2]{SVV}, it follows that
	\[\H^1\DB_X^2\neq 0.\]
(Note that for a counterexample we needed $\H^i \DB_X^2\neq 0$ for some $i \le 3$.)
\end{example}

In view of these examples, and of the results of this paper, in retrospect the question on vanishing in \cite{MP-LC} should have been more restrictive. Namely, is it true that for $X$ arbitrary we have
\begin{equation}\label{eqn:modified}
\H^i\DB_X^k = 0 \quad\text{ for }\quad 0<i< n-k-1- {\rm lcdef}(X) - s?
\end{equation}
It turns out however that even this statement is false, again already for cones over special smooth varieties. In this case (or whenever the singularities are isolated) thanks to Corollary \ref{cor:depth-DB-complex} the question becomes whether
\[\H^i\DB_X^k = 0 \,\,\,\,\,\, \text{for} \,\,\,\, 0 <  i< \min_{p\ge 0}\{\depth \DB_X^p + p\}-k-1.\]

\begin{example}
Let $Y$ be a smooth projective variety with $\dim Y \ge 3$ and $H^1 (Y, \O_Y) = 0$, endowed with a very ample line bundle $L$ such that 
$H^1(X, L) \neq 0$. We will show the existence of such a variety below; for now we draw some conclusions about the abstract cone $X = C(Y, L)$.

We claim that the modified question in ($\ref{eqn:modified}$) has a negative answer when $k=0$ and $i=1$. 
First, by \cite[Proposition 7.2]{SVV}, the hypothesis $H^1 (Y,L) \neq 0$ implies that  $\H^1\DB_X^0\neq 0$. On the other hand, the computation of the depth of Du Bois complexes of cones is addressed in \cite[Theorem 3.1(2)]{PSh}. In our example it implies that:
\begin{itemize}
\item $\depth \DB_X^2>0$
\smallskip
\item  $\depth \DB_X^1>1 \iff$  $\begin{cases} 
H^1(Y, \Omega_Y^1\otimes L^m)=0 \text{ for } m\le -1 \\
	 	 H^1(Y, \O_Y)=H^0(Y, \Omega_Y^1)=0 \\
	 	 H^0(Y, \O_Y)\overset{\cup c_1(L)}\longrightarrow H^1(Y, \Omega_Y^1) \text{ is injective}
	 \end{cases}$
\smallskip	 
\item	 $\depth \DB_X^0>2 \iff $  $\begin{cases}
	 	 H^0(Y, L^m)=0 \text{ for } m\le -1 \\
	 	H^1(Y, L^m)=0 \text{ for } m\le 0
	 \end{cases}$
\end{itemize}	 
	All of these conditions, other than $H^1(Y, L)\neq 0$ and $H^1(Y, \O_Y)=0$ provided by the hypothesis, are satisfied by Kodaira-Nakano vanishing and Hard Lefschetz.  It follows that
	\[\min_{p\ge 0}\{\depth \DB_Z^p + p \}-1 >1,\]
	while $\H^1\DB_Z^0\neq 0$.

Here is an example of a threefold $Y$ satisfying the required properties: let $Y\subset \PP^2\times \PP^1\times \PP^1$ be a general hypersurface 
in the linear system $|\O_{\PP^2}(-d)\boxtimes \O_{\PP^1}(-1) \boxtimes \O_{\PP^1}(-1)|$, with $d\gg 0$, and let $L=\O_Y (1,1,1)$. 
Then chasing cohomology through the short exact sequences
	\[0\to \O_{\PP^2\times \PP^1\times \PP^1}(-d,-1,-1)\to \O_{\PP^2\times \PP^1\times \PP^1}\to \O_Y\to 0\]
	\[0\to \O_{\PP^2\times \PP^1\times \PP^1}(-d+1,0,0)\to \O_{\PP^2\times \PP^1\times \PP^1}(1,1,1)\to L\to 0\]
	and using the  K\"unneth formula easily shows that $H^1 (Y, \O_Y) = 0$ and $H^1 (Y, L) \neq 0$.
	
Examples of any dimension can be obtaines as follows: take products $Y \times Z$ and line bundles $L \boxtimes M$, where $Y$ is the variety above, and $Z$ is such that $H^1(Z, \O_Z) = 0$ and has an ample line bundle $M$ with $H^0 (Z, M) \neq 0$.

\medskip

As a general conclusion to this section, the answer to the question regarding which higher cohomologies $\H^i \DB_X^k$ vanish is dictated by the depth of $\H^0 \DB_X^k$, which can sometimes be smaller than $n - k - {\rm lcdef (X)}$.
\end{example}

\subsection{$k$-rational implies $k$-Du Bois}\label{scn:krat-kDB}
Theorem \ref{thm:injectivity-H0DB} leads to a very quick alternative proof of the fact that normal, pre-$k$-rational isolated singularities are pre-$k$-Du Bois, obtained (even for non-isolated singularities) in \cite[Theorem B]{SVV}. As explained in \cite[Corollary 5.8]{SVV}, it then follows easily that $k$-rational implies $k$-Du Bois, in the same setting. Recall that when $k = 0$ this implication was studied in  \cite[Proposition 3.7]{Steenbrink-isolated-rational-implies-DB} for isolated singularities, and in \cite{Kovacs-rational-implies-DB}, \cite{Saito-MHC} in general. Later, \cite[Theorem 1.6]{FL}  and \cite[Theorem B]{MP-lci}  proved that $k$-rational implies $k$-Du Bois for local complete intersections.

If  $X$ is normal and pre-$k$-rational, then it has rational singularities.  Therefore, using the main result of \cite{KS-extending}, one has that the composition
 \[\H^0\DB_X^k \to  \DB_X^k  \to  \DD (\DB_X^{n-k})\]
 is a quasi-isomorphism; see \cite[Remark 2.5]{SVV} for details. Here we denote 
 $$\DD_X (\DB_X^{n-k}) := \RHom_{\O_X}(\DB_X^{n-k},\omega_X^{\bullet}[-n]).$$

Dualizing, this gives that the 
composition
\[\RHom_{\O_X}( \DD (\DB_X^{n-k}),\omega_X^{\bullet})\longrightarrow \RHom_{\O_X}(\DB_X^k,\omega_X^{\bullet})\overset{\varphi}{\longrightarrow} \RHom(\H^0\DB^k_{X},\omega_X^{\bullet})\]
is a quasi-isomorphism as well, which in turn implies that $\varphi$ is surjective on cohomology. By induction we may assume however that $X$ has 
pre-$(k-1)$-Du Bois singularities, hence it is also injective on cohomology by Theorem \ref{thm:injectivity-H0DB}. It follows that $\varphi$ 
is a quasi-isomorphism, and dualizing again we obtain that $X$ is pre-$k$-Du Bois. Conjecture \ref{conj:injectivity-H0DB} would of course make the same proof work even in the non-isolated case.

\section{Analogues for the intersection complex}\label{ch:IC}

\subsection{On the relationship between Du Bois and intersection complexes}
Let $X$ be a complex variety of dimension $n$. Recall from \cite[Section 4.5]{Saito-MHM} that we have an object $\QQ^H_X[n] : = a_X^* \QQ^H_{\rm pt}$ in the derived category of mixed Hodge modules on $X$, with cohomologies in degrees $\le 0$; moreover, the top degree in the weight filtration on $\H^0  \QQ^H_X[n]$ is $n$. We also have the intersection complex $\IC_X\QQ^H$, a simple pure Hodge module of weight $n$; moreover, there is a composition of quotient morphisms
$$\gamma_X: \QQ^H_X[n]\to  \H^0  \QQ^H_X[n]\to \IC_X\QQ^H \simeq \gr^W_n  \H^0  \QQ^H_X[n].$$

We first record the following simple fact for later use. Here $\DDD_X (-)$ denotes the duality functor on the derived category of filtered $D$-modules underlying mixed Hodge modules (see \cite[Section 2.4]{Saito-PMHM}), and we abuse the notation by continuing to use the Hodge module notation for the respective filtered $D$-modules. Moreover $M (\ell)$ denotes the Tate twist of $M$, which at the level of filtered $D$-modules shifts the filtration down by $\ell$, i.e. 
$F_\bullet M(\ell) = F_{\bullet - \ell} M$.

\begin{lemma} \label{lemma: self dual of Q^H_X and rational homology}
    For a fixed integer $k$, the composition
    \[F_k \QQ^H_X[n] \to F_k \IC_X \QQ^H \to F_k \DDD_X(\QQ^H_X[n])(-n)\]
    is an isomorphism if any only if
    \[F_k \QQ^H_X[n] \to F_k \IC_X \QQ^H\]
    and 
    \[F_k \IC_X \QQ^H \to F_k \DDD_X(\QQ^H_X[n])(-n)\]
    are both isomorphisms.
\end{lemma}
\begin{proof}
The ``if" part is obvious, so we focus on the ``only if" part. Moreover, it suffices to prove that the first map is an isomorphism.

By strictness, the Hodge filtration commutes with taking cohomology, hence we have that the composition 
    \[F_k (\H^i \QQ^H_X[n]) \to F_k (\H^i \IC_X \QQ^H) \to F_k \big(\H^i \DDD_X(\QQ^H_X[n])(-n)\big)\]
    is an isomorphism for all $i \in \ZZ$. Thus for $i\neq 0$ we get 
    \[F_k (\H^i \QQ^H_X[n]) = \H^i(F_k \QQ^H_X[n]) = 0.\]
     When $i = 0$ we obtain an isomorphism 
    \[F_k (\H^0 \QQ^H_X[n]) \to  F_k \IC_X \QQ^H \to  F_k \big(\H^0 \DDD_X(\QQ^H_X[n])(-n)\big),\]
which implies that the first map is an injection.
   On the other hand, the fact that $ \IC_X\QQ^H \simeq \gr^W_n  \H^0  \QQ^H_X[n]$ implies that the morphism $\H^0  \QQ^H_X[n]\to \IC_X\QQ^H$ is 
   surjective at the level of  filtered $D$-modules. Putting everything together, we obtain isomorphisms
    \[F_k \QQ^H_X[n] \simeq F_k (\H^0 \QQ^H_X[n]) \simeq F_k \IC_X \QQ^H \]
   which implies what we want.
   \end{proof}

For what follows, recall that we use the notation
 $$\DD_X (-) := \RHom_{\O_X}(-,\omega_X^{\bullet}[-n]).$$
We will make repeated use of the well-known commutation of the two duality functors via the graded de Rham functor, proved in \cite[Section 2.4]{Saito-PMHM}: if $M^\bullet$ is an object in $D^b{\rm MHM} (X)$, and $p$ is any integer, then:
\begin{equation}\label{eqn:duality-MHM}
\DD_X \big(\gr^F_p \DR (M^\bullet)\big) \simeq \gr^F_{-p} \DR \big(\DDD_X (M^\bullet)\big) [-n].
\end{equation}

\medskip

It is a consequence of \cite[Theorem 4.2]{Saito-MHC} that for each $p$, we have the identification 
$$\DB^p_X \simeq \gr^F_{-p}\DR (\QQ_X^H [n] )[p-n].$$
We introduce the following notation for simplicity:
$$I\DB_X^p :=  \gr^F_{-p}\DR (\IC_X\QQ^H )[p-n].\footnote{These objects are called \emph{intersection Du Bois complexes} in \cite{PP}, where they are used extensively.}$$
Taking the composition of the morphism $\gamma_X$ with its dual, we get the natural morphisms in the derived category of mixed Hodge modules $D^b{\rm MHM}(X)$:
$$\QQ^H_X[n] \to \IC_X\QQ^H \to (\DDD_X(\QQ^H_X[n]))(-n)$$
due to the self-duality $\DDD_X (\IC_X\QQ^H) \cong \IC_X\QQ^H(n)$; see \cite[4.5.13]{Saito-MHM}. Applying the functor $\gr^F_{-p} \DR$ to this composition, we obtain the natural morphisms 
\begin{equation}\label{DB-to-IC}
\DB^p_X \xto{\varphi_p} I\DB^p_X \to \DD_X(\DB^{n-p}_X).
\end{equation}
in the derived category of coherent sheaves on $X$. 

An important point is that the higher rationality conditions say something about these maps. When $X$ is a local complete intersection, this is essentially contained 
in the proof of \cite[Theorem 3.1]{CDM}. 

\begin{proposition}\label{prop:rat_IC}
Let $X$ be a normal variety with pre-$k$-rational singularities. Then $\varphi_p$ is an isomorphism for all $p \le k$.
\end{proposition}
\begin{proof}
As discussed in Section \ref{scn:krat-kDB}, if $X$ is normal with pre-$k$-rational singularities, of dimension $n$,  
then the natural morphisms
$$\DB_X^p \to \DD_X (\DB_X^{n-p})$$
are isomorphisms for $p \le k$. Since 
    \[\DB^p_X \simeq \gr^F_{-p} \DR (\QQ^H_X[n])[p-n],\]
    and  moreover
    \begin{align*}
    \DD_X(\DB^{n-p}_X) &\simeq \DD_X(\gr^F_{p- n} \DR (\QQ^H_X[n])[-p]) \\ 
    &\simeq \gr^F_{n-p} \DR (\DDD_X(\QQ^H_X[n]))[p-n] \\
    &\simeq \gr^F_{-p} \DR \big (\DDD_X(\QQ^H_X[n])(-n)\big )[p- n]
    \end{align*}
    we obtain
    \[\gr^F_{-p} \DR(\QQ^H_X[n]) \simeq \gr^F_{-p} \DR \big (\DDD_X(\QQ^H_X[n])(-n)\big )\]
    for all $p\le k$. After dualizing via ($\ref{eqn:duality-MHM}$), this is equivalent to 
    \[\gr^F_{p} \DR(\QQ^H_X[n]) \simeq \gr^F_{p} \DR \big (\DDD_X(\QQ^H_X[n])(-n)\big )\]
    for all $p\le k-n$. According to the general Lemma \ref{lem:MHM-general} below, this is equivalent to 
    \[F_{p} \QQ^H_X[d_X] \simeq F_{p} \DDD_X(\QQ^H_X[n])(-n)\]
    for all $p\le k- n$. Lemma \ref{lemma: self dual of Q^H_X and rational homology} implies in turn 
    \[F_p \IC_X \QQ^H \simeq F_{p} \DDD_X(\QQ^H_X[n])(-n)\]
    for all $p\le k-n$, which again by Lemma \ref{lem:MHM-general} is equivalent  to 
    $$\gr^F_p \DR (\IC_X\QQ^H ) \simeq \gr^F_p \DR \big(\DDD_X(\QQ^H_X[n])(-n)\big)$$
    for $p \le k- n$. Applying ($\ref{eqn:duality-MHM}$) one more time, we see that
    \[\DB^p_X \simeq I\DB^p_X\]
    for all $p\le k$.
\end{proof}

\begin{remark}
Note that consequences of the isomorphisms in Proposition \ref{prop:rat_IC} regarding the topology of $X$ are studied in the upcoming \cite{DOR} and \cite{PP}. 
In particular, it is shown in these papers that if $\varphi_p$ is an isomorphism for all $p \le \lceil (n - 2)/2 \rceil$, then  $X$ is a rational homology manifold.
\end{remark}

The following useful lemma is a rather straightforward application of the definitions and the strictness of the Hodge filtration.

\begin{lemma}\label{lem:MHM-general}
 Let $M^\bullet, N^\bullet  \in {\rm D}^b {\rm MHM} (X)$ be objects in the bounded derived category of mixed Hodge modules on $X$. Then the following are equivalent:
    \begin{enumerate}
        \item $\gr^F_p \DR(M^\bullet)$ and  $\gr^F_p \DR(N^\bullet)$ are quasi-isomorphic for all $p\le k$.
        \item $F_p M^\bullet$ and $F_p N^\bullet$ are quasi-isomorphic for all $p\le k$. 
    \end{enumerate}
\end{lemma}

If $X$ is (locally) embedded into a smooth variety $Y$ of dimension $d$, we have
$$ \gr^F_{p} \DR(M^\bullet) =   \left[\gr^F_{p-d}M^\bullet  \otimes \wedge^{d} T_Y \to \dots \to \gr^F_{p-1} M^\bullet \otimes T_Y \to 
\gr^F_{p} M^\bullet \right],$$
placed in degrees $-d$ to $0$. Here, and in the rest of this paper, we use this notation to denote the total complex associated to the double complex where the vertical 
maps come from the differentials of a representative of $M^\bullet$, while the horizontal maps are the usual de Rham maps on each term of that representative. The Lemma then follows  by induction on $k$.

\subsection{Injectivity results and conjecture for the intersection complex}\label{scn:injectivity-IC}
We start with an injectivity conjecture which is the intersection complex analogue of the main Conjecture \ref{conj:injectivity-H0DB}.

\begin{conjecture}\label{conj:IC2}
If $X$  has normal and pre-$(k-1)$-rational singularities, the natural morphism
$$\RHom(I\DB_X^k,\omega_X^{\bullet}) \to \RHom(\H^0 I\DB_X^k,\omega_X^{\bullet})$$
obtained by dualizing the canonical morphism $\H^0 I\DB_X^k \to I\DB_X^k$ is injective on cohomology.
\end{conjecture}

\begin{remark}\label{rmk:H0-IC}
It is shown in \cite[Proposition 8.1]{KS-extending} that for all $k$ we have an isomorphism 
$$\H^0 I\DB_X^k \simeq f_* \Omega_{\widetilde X}^k,$$
where $f \colon \widetilde{X} \to X$ is a resolution of singularities.
\end{remark}

Note that there is a commutative diagram 
    \[\begin{tikzcd}
        &\RHom(I\DB_X^k,\omega_X^{\bullet})\ar[d] \ar[r]& \RHom(\DB_X^k,\omega_X^{\bullet}) \ar[d] \\
        &\RHom(\H^0 I\DB_X^k,\omega_X^{\bullet})\ar[r]& \RHom(\H^0 \DB_X^k,\omega_X^{\bullet}).
    \end{tikzcd}\]
Moreover, if $X$ has normal pre-$(k-1)$-rational singularities, then it also has pre-$(k-1)$-Du Bois singularities 
by \cite[Theorem B]{SVV}. Therefore Conjecture \ref{conj:IC2} is in fact implied by Conjecture \ref{conj:injectivity-H0DB}, thanks to the following:

\begin{theorem}\label{thm:IC1}
If $X$ has normal pre-$(k-1)$-rational singularities, the morphism
$$\RHom(I\DB_X^k,\omega_X^{\bullet}) \to \RHom(\DB_X^k,\omega_X^{\bullet})$$
obtained by dualizing $\varphi_k$ is injective on cohomology. 
\end{theorem}

\begin{corollary}\label{cor:conj-IC}
Conjecture \ref{conj:IC2} holds in any of the following cases:
\begin{enumerate}
\item $k =0$.
\item $X$ has isolated singularities.
\item $X$ is a local complete intersection with $(k-1)$-rational singularities.
\end{enumerate}
\end{corollary}
\begin{proof}
In each of these cases we have the corresponding injectivity theorem for the Du Bois complex, answering Conjecture \ref{conj:injectivity-H0DB}:
for $k = 0$ by \cite{KS-injectivity}, for isolated singularities by Theorem \ref{thm:injectivity-H0DB} here, and for local complete intersections by \cite{MP-lci} (note that 
$(k-1)$-rational singularities are $(k-1)$-Du Bois).
\end{proof}

Theorem \ref{thm:IC1} is in turn a consequence of Proposition \ref{prop:rat_IC}, combined with the following injectivity theorem, which is 
the main technical result of this section. It was first communicated to us by Sung Gi Park, whom we thank, using the technique of \cite[Lemma 3.7]{Park}; we follow the method of the previous section.


\begin{theorem}[{Sung Gi Park}]\label{thm:injectivity-Rk}
Assume that the variety $X$ satisfies the property that $\varphi_p \colon \DB_X^p \to I\DB_X^p$ is an isomorphism for $p \le k-1$. 
Then the morphism
$$\RHom_{\O_X}(I\DB_X^k,\omega_X^{\bullet}) \to \RHom_{\O_X}(\DB_X^k,\omega_X^{\bullet})$$
obtained by dualizing $\varphi_k$ is injective on cohomology. 
More precisely, it is an isomorphism on $i$-th cohomology for $i \le k-n - 1$, injective for $i= k- n$, and $\sExt^i_{\O_X}(I\DB_X^k,\omega_X^{\bullet}) = 0$ for $i >k - n$.
\end{theorem}
\begin{proof}  
By (\ref{eqn:duality-MHM}), we have $\DD_X (I\DB_X^k) \simeq I\DB_X^{n-k}$, so
$$\E xt^i_{\O_X}(I\DB_X^k, \omega_X^{\bullet}) \cong \H^{i+n} I\DB_X^{n-k}  =  0 \,\,\,\,\,\,{\rm for}\,\,\,\, i > k - n.$$
Thus we focus on the statements for $i \le k -n$.

As in the proof of Proposition \ref{prop:rat_IC}, the assumption implies that we have
isomorphisms 
\begin{equation}\label{eqn:HF}
	F_{p} \IC_X\QQ^H  \simeq F_{p} \DDD_X(\QQ^H_X[n])(-n)  \,\,\,\,\,\,{\rm for}\,\,\,\, p \le k-1-n.
\end{equation}


Using (\ref{eqn:duality-MHM}), the conclusion is equivalent to the fact that the map
    $$\gr^F_{k-n} \DR (\IC_X\QQ^H ) \to \gr^F_{k-n} \DR \big(\DDD_X(\QQ^H_X[n])(-n)\big)$$
is an isomorphism on $i$-th cohomology for $i < 0$, and injective for $i =0$.

To simplify the notation, we set $M^\bullet := \DDD_X(\QQ^H_X[n])(-n)$. We also think of $X$ as being (locally) embedded in a smooth 
variety $Y$ of dimension $d$, so that we have 
$$ \gr^F_{k-n} \DR(M^\bullet) =   \left[\gr^F_{k-n-d}M^\bullet  \otimes \wedge^{d} T_Y \to \dots \to \gr^F_{k-n-1} M^\bullet \otimes T_Y \to 
\gr^F_{k-n} M^\bullet \right],$$
placed in degrees $-d$ to $0$. In other words, we have an exact triangle
$$\gr^F_{k-n} M^\bullet \to \gr^F_{k-n} \DR(M^\bullet) \to A^\bullet \xto{+1},$$
where 
$$A^\bullet : =   \left[\gr^F_{k-n-d}M^\bullet  \otimes \wedge^{d} T_Y \to \dots \to \gr^F_{k-n-1} M^\bullet \otimes T_Y\right],$$
placed in degrees $-d$ to $-1$.
Note that ($\ref{eqn:HF})$ implies the isomorphism  
 \[   \left[\gr^F_{k-n-d} \IC_X \QQ^H \otimes \wedge^{d} T_Y \to \dots \to \gr^F_{k-n-1} \IC_X \QQ^H \otimes T_Y\right]  \overset{\simeq}{\longrightarrow} A^\bullet,\]
    hence we also have a similar exact triangle with $M^\bullet$ replaced by $\IC_X \QQ^H$.

   %
   
    
  Consider now the exact triangle 
    \[\gr^F_{k-n} \DR(\IC_X \QQ^H) \to \gr^F_{k-n} \DR(M^\bullet) \to C^\bullet \xto{+1}\]
    where $C^\bullet$ denotes the cone of the morphism on the left. Since $\gr^F_{k-n}\DR(\IC_X \QQ^H)$ is a complex placed in non-positive degrees, we only need to show that 
    \[\H^i C^\bullet  = 0 \,\,\,\,\,\,{\rm for} \,\,\,\, i<0.\]
    
    Given the triangles described above, the octahedral axiom implies that 
    we have an exact triangle 
 \[\gr^F_{k-n} \IC_X \QQ^H \to \gr^F_{k-n} M^\bullet \to C^\bullet \xto{+1}\]   
as well. Hence the needed statement about $\H^i C^\bullet$ follows immediately from the following facts.
On the one hand, 
\[\H^i \gr^F_{k-n} M^\bullet \simeq \gr^F_{k-n} \H^i M^\bullet = 0 \,\,\,\,\,\,{\rm for} \,\,\,\, i < 0, \]
 since $M^\bullet$ has nontrivial cohomologies only in non-negative degrees (being essentially the dual of $\QQ_X^H [n]$, which has nontrivial 
 cohomologies in non-positive degrees). On the other hand, $\H^i \IC_X \QQ^H = 0$ for $i \neq 0$, and $F_{k - n} \IC_X \QQ^H \hookrightarrow 
 F_{k-n} \H^0 M^\bullet$ (while at the level of $F_{k-n-1}$ we have equality by ($\ref{eqn:HF}$)). Indeed, we have seen in the proof of Lemma \ref{lemma: self dual of Q^H_X and rational homology} that the morphism 
$\H^0 \QQ_X^H \to  \IC_X \QQ^H$ is surjective at the level of filtered $D$-modules; similarly, by duality, we have that the 
morphism $\IC_X \QQ^H \to \H^0 M^\bullet$ is injective at the level of filtered $D$-modules. All of this implies that we have an injection
\[\gr^F_{k-n} \IC_X \QQ^H \to \gr^F_{k-n} \H^0 M^\bullet.\]
This completes the proof.  
\end{proof}

\subsection{Vanishing of higher cohomologies for intersection complexes}\label{scn:vanishing-IC}
We finish with results about the vanishing of higher cohomologies of $I\DB_X^k$. 
We start with the following proposal, analogous to Conjecture \ref{conj:vanishing-general}:

\begin{conjecture}\label{conj:vanishing-general-IC}
 Let $X$ be a normal variety with pre-$(k-1)$-rational singularities, and pre-$k$-rational away from a closed subset of dimension $s$. Then
 \[\H^i I \DB^k_X = 0 \quad \text{ {\rm for} } \quad 0<i<\depth\H^0 I \DB^k_X- s -1.\] 
\end{conjecture}

When $X$ is a normal variety with pre-$k$-rational singularities, it follows from Proposition \ref{prop:rat_IC} that $\H^i I\DB_X^p = 0$ for all $i > 0$ and $p \le k$. If it is so only away from a closed set of dimension $s$, then  an argument completely analogous to that of Proposition \ref{prop:vanishing-general} implies that 
\begin{equation}\label{eqn:depth-IC}
 \H^i I\DB^k_X = 0 \quad \text{ {\rm for} } \quad 0<i< n_k - s -1,
 \end{equation}
 where $n_k := {\rm min}~ \{\depth\H^0 I \DB^k_X, ~\depth I\DB^k_X + 1\}$.

If moreover $X$ has normal pre-$(k-1)$-rational isolated or $(k-1)$-rational local complete intersection singularities,  or if $k = 0$, then due to Corollary \ref{cor:conj-IC} we have that Conjecture \ref{conj:IC2} holds, hence exactly as in the proof of Theorem \ref{thm:vanishing-isolated} we have in addition that $\depth  I \DB_X^k\ge \depth \H^0 I \DB_X^k$. Therefore we obtain:

\begin{corollary}\label{cor:IC-van}
Conjecture \ref{conj:vanishing-general-IC} holds when $X$ has isolated,  or $(k-1)$-rational local complete intersection singularities, or when $k = 0$.
\end{corollary}

We also have the analogue of the vanishing result for the Du Bois complex in \cite[Corollary 13.9]{MP-LC}; see Example \ref{ex:vanishing-lci}.

\begin{corollary}\label{cor:IC-vanishing-LCI}
 Let $X$  be a local complete intersection with $\dim X_{\rm sing} = s$. Then
  $$ \H^i I \DB_X^k = 0, \,\,\,\,\,\,{\rm for ~all} \,\,\,\, 0 <  i < n -k -s -1.$$ 
\end{corollary}
\begin{proof} 
We use ($\ref{eqn:depth-IC}$). Note first that for the intersection complex, it is always the case that $\depth I \DB_X^k\ge n-k$  by ($\ref{eqn:depth}$). Indeed, we have
$$\E xt^i_{\O_X}(I\DB_X^k, \omega_X^{\bullet})  =  0 \,\,\,\,\,\,{\rm for}\,\,\,\, i > k - n$$
due to the self-duality (up to twist) of the intersection complex, as explained in Theorem \ref{thm:injectivity-Rk}.

Moreover, following precisely the steps in  Example \ref{ex:vanishing-lci}, under the current hypotheses we also have that 
$\depth \H^0  I \DB_X^k\ge n-k$;  indeed, if $f\colon \widetilde{X} \to X$ is a resolution of singularities, we know that $ \H^0  I \DB_X^k \simeq f_* \Omega^k_{\widetilde X}$ 
(see Remark \ref{rmk:H0-IC}), which in this case is reflexive, 
so that we have 
$$ f_* \Omega_{\widetilde X}^k \simeq \Omega_X^{[k]} \simeq \Omega_X^k.$$
\end{proof}



\begin{thebibliography}{99}



\bibitem[CDM]{CDM}
Q. Chen, B. Dirks, and M. Musta\c{t}\u{a}, \emph{The minimal exponent and $k$-rationality for locally complete
 intersections}, preprint arXiv:2212.01898.

\bibitem[DOR]{DOR}
B. Dirks, S. Olano and D. Raychaudhury, in preparation.

\bibitem[DT]{DT}
H. Dao and S. Takagi, \emph{On the relationship between depth and cohomological dimension}, Compositio Math \textbf{152} (2016), 876--888.

\bibitem[DB]{DB}
P. Du Bois,  \emph{Complexe de de Rham filtr\'e d'une vari\'et\'e singuli\`{e}re}, Bull. Soc. Math. France \textbf{109} (1981), no. 1, 41–81.


\bibitem[FY]{FY}
H.-B. Foxby and S. Iyengar, \emph{Depth and amplitude for unbounded complexes}, Commutative algebra (Grenoble/Lyon, 2001), 119--137, Contemp. Math.  \textbf{331}, Amer. Math. Soc., Providence, RI, 2003.

\bibitem[FL1]{FL-isolated}
R. Friedman and R.Laza, \emph{ The higher Du Bois and higher rational properties for isolated singularities}, 
 J. Algebraic Geom. \textbf{33} (2024), no. 3, 493--520.
 
 \bibitem[FL2]{FL}
 R. Friedman and R.Laza, \emph{ Higher Du Bois and higher rational singularities}, 
 preprint arXiv:2205.04729, to appear in Duke Math. J.

\bibitem[Gra]{Graf}
P. Graf, \emph{The generalized {L}ipman-{Z}ariski problem}, Math. Ann. \textbf{362} (2015), no. 1-2, 241--264.

\bibitem[Gre]{Greuel}
G.-M. Greuel, \emph{ Der {G}auss-{M}anin-{Z}usammenhang isolierter {S}ingularit\"{a}ten
  von vollst\"{a}ndigen {D}urchschnitten},  Math. Ann. \textbf{214} (1975), 235--266.

\bibitem[GNPP]{GNPP}
F. Guill\'en, V. Navarro Aznar, P. Pascual Gainza, and F. Puerta, \emph{Hyperr\'esolutions cubiques et descente cohomologique}, Lecture Notes in Mathematics, vol. 1335, Springer-Verlag, Berlin, 1988. Papers from the Seminar on Hodge-Deligne Theory held in Barcelona, 1982.


\bibitem[HJ]{h-differential}
A. Huber and C. J\"{o}rder, \emph{ Differential forms in the h-topology},  Algebr. Geom. \textbf{1} (2014), no. 4, 449--478.

\bibitem[JKSY]{JKSY}
S.-J. Jung, I.-K. Kim, M. Saito, and Y. Yoon, \emph{Higher Du Bois singularities of hypersurfaces}, Proc. Lond. Math. Soc. \textbf{125} (2022), no. 3, 543--567.

\bibitem[Ko99]{Kovacs-rational-implies-DB}
S. J. Kov\'{a}cs, \emph{ Rational, log canonical, Du Bois singularities: on the conjectures of Koll\'{a}r and Steenbrink},  Compositio Math. \textbf{118} (1999), no. 2, 123–133.

\bibitem[KS16]{KS-injectivity}
S. J. Kov\'{a}cs and K. Schwede, \emph{ Du {B}ois singularities deform}, in {\em Minimal models and extremal rays ({K}yoto, 2011)}, volume~70
  of {\em Adv. Stud. Pure Math.}, pages 49--65. Math. Soc. Japan, [Tokyo], 2016.

\bibitem[KS21]{KS-extending}
S. Kebekus and Ch. Schnell, \emph{ Extending holomorphic forms from the regular locus of a complex space
  to a resolution of singularities}, J. Amer. Math. Soc. \textbf{34} (2021), no. 2, 315--368.

\bibitem[MV]{MV}
C. Miller and S. Vassiliadou, \emph{ ({C}o)torsion of exterior powers of differentials over complete
  intersections}, J. Singul. \textbf{19} (2019), 131--162.

\bibitem[MOPW]{MOPW}
M. Musta\c{t}\u{a}, S. Olano, M. Popa, and J. Witaszek, \emph{The {D}u {B}ois complex of a hypersurface and the minimal exponent}
Duke Math. J. \textbf{172} (2023), no. 7, 1411--1436.

\bibitem[MP1]{MP-LC}
M. Musta\c{t}\u{a} and M. Popa, \emph{ Hodge filtration on local cohomology, {D}u {B}ois complex and local
  cohomological dimension}, Forum Math. Pi  \textbf{10} (2022), Paper No. e22, 58.

\bibitem[MP2]{MP-lci}
M. Musta\c t\u a and M. Popa, \emph{ On k-rational and k-{D}u {B}ois local complete intersections}, preprint arXiv:2207.08743, to appear in 
Algebraic Geometry.

\bibitem[NS]{NS}
Y. Namikawa and J.~H.~M. Steenbrink, \emph{ Global smoothing of Calabi-Yau threefolds}, 
Invent. Math. \textbf{122} (1995), 403--419.

\bibitem[Og]{Ogus}
A. Ogus, \emph{Local cohomological dimension of algebraic varieties}, Ann. of Math. \textbf{98}  (1973), no. 2, 327--365.

\bibitem[Pa]{Park}
S. G. Park \emph{Du Bois complex and extension of forms beyond rational singularities}, 
preprint arXiv:2311.15159  (2023).

\bibitem[PP]{PP}
S. G. Park and M. Popa, \emph{Lefschetz theorems, $\QQ$-factoriality, and Hodge symmetry for singular varieties}, in preparation.

\bibitem[PS]{PS}
C. Peters and J. Steenbrink, \emph{Mixed Hodge structures}, Ergebnisse der Mathematik und ihrer Grenzgebiete. 3. Folge, vol. 52, Springer-Verlag, Berlin, 2008.

\bibitem[PSh]{PSh}
M. Popa and W. Shen, \emph{Du Bois complexes of cones over singular varieties, local cohomological dimension, and $K$-groups}, 
preprint arXiv:2406.03593  (2024).

\bibitem[Sa1]{Saito-PMHM}
M. Saito, \emph{Modules de {H}odge polarisables}, Publ. Res. Inst. Math. Sci. \textbf{24} (1988), no. 6, 849--995.

\bibitem[Sa2]{Saito-MHM}
M. Saito, \emph{Mixed {H}odge modules}, Publ. Res. Inst. Math. Sci. \textbf{26} (1990), no. 2, 221--333.

\bibitem[Sa3]{Saito-MHC}
M. Saito, \emph{Mixed {H}odge complexes on algebraic varieties},  Math. Ann. \textbf{316} (2000), no. 2, 283--331.

\bibitem[St1]{Steenbrink-isolated-rational-implies-DB}
J.~H.~M. Steenbrink, \emph{ Mixed Hodge structures associated with isolated singularities}, 
Singularities, Part 2 (Arcata, Calif., 1981), 513–536.

\bibitem[St2]{Steenbrink-vanishing}
J.~H.~M. Steenbrink, \emph{ Vanishing theorems on singular spaces}, in \emph{ Differential systems and singularities} (Luminy, 1983), Ast\'{e}risque \textbf{130} (1985), 330--341.

\bibitem[St3]{Steenbrink-DB-invariants}
J.~H.~M. Steenbrink, \emph{ Du {B}ois invariants of isolated complete intersection singularities}, 
Ann. Inst. Fourier (Grenoble) \textbf{47} (1997), no. 5, 1367--1377.

\bibitem[SVV]{SVV}
W. Shen, S.Venkatesh, and A.~D. Vo,  \emph{ On $k$-Du Bois and $k$-rational singularities}, preprint arXiv:2306.03977.

\bibitem[Tig]{Tighe}
B. Tighe, \emph{The holomorphic extension property for $k$-Du Bois singularities}, preprint arXiv:2312.01245.

\bibitem[Tot]{Totaro}
B. Totaro, \emph{Bott vanishing for Fano 3-folds}, Math. Z. \textbf{307} (2024), no. 1, 14--31.


\end{thebibliography}
\end{document}